\documentclass[12pt, leqno]{article}
\usepackage{amsmath,amsfonts,amsthm,amscd,amssymb,verbatim}
\usepackage{color, ulem}
\usepackage[all]{xy}

\numberwithin{equation}{section}
\usepackage[english]{babel}
\usepackage{tikz-cd}
\usepackage{color}

\newcommand{\cl}{\mathcal}
\newcommand{\op}{\operatorname}
\newcommand{\llog}{\operatorname{log}}
\newcommand{\llogcr}{\operatorname{logcrys}}
\newtheorem{theorem}{Theorem}[section]

\newtheorem{remark}[theorem]{Remark}
\newtheorem{definition}[theorem]{Definition}

\newtheorem{lemma}[theorem]{Lemma}
\newtheorem{corollary}[theorem]{Corollary}

\newenvironment{claim}[1]{\par\noindent\textbf{Claim:}\space#1}{}

\usepackage{hyperref}

\title{A note on rank 1 log extendable isocrystals on non-proper varieties}
\date{}
\author{Efstathia Katsigianni\footnote{Freie Universit\"at Berlin, Arnimallee 3, 14195, Berlin, Germany.  Supported by Berlin Mathematical School.  Email: \texttt{katsief@zedat.fu-berlin.de}.}} 
\begin{document}
\maketitle
\begin{abstract}
In 2010 de Jong proposed a $p$-adic version of Gieseker's conjecture: if  $\pi_1^{\text{\'et}}(X)=1$, for $X$ smooth connected projective variety, then any isocrystal on $X$ is constant. This was proven by Esnault and Shiho \cite{ES1},\cite{ES2} under some additional assumptions. We show that the conjecture holds in the case of a non-proper variety with trivial tame fundamental group and  rank 1 log extendable isocrystals.
\end{abstract}  

\section{Introduction}
  
For a complex connected smooth projective variety $X$ one can define the \'etale fundamental group $\pi_1^{\text{\'et}}(X)$, which classifies all the finite \'etale coverings of $X$ and is isomorphic to the profinite completion of its topological fundamental group. On the other hand we have the category of $\mathcal O_X$-coherent $\mathcal D_X$-modules, which is equivalent via the Riemann-Hilbert correspondence to the category of finite dimensional representations of the topological fundamental group.  By a result of Mal\v{c}ev \cite{mal} and Grothendieck \cite{gro} we know how this category relates to  $\pi_1^{\text{\'et}}(X)$: if  $\pi_1^{\text{\'et}}(X)$ is trivial, then there are no non-constant $\mathcal O_X$-coherent $\mathcal D_X$-modules.

Naturally, one asks if the same holds true in positive characteristic. This conjecture was formulated by Gieseker \cite{gie} and proven by Esnault and Mehta in \cite{esnmeh}. 

In 2010 de Jong proposed a $p$-adic version of this conjecture: if  $\pi_1^{\text{\'et}}(X)=1$, then any isocrystal $\mathcal E\in \operatorname{Crys}(X/W)_{\mathbb Q}$ is constant.

This question was addressed in great detail in \cite{ES1} and in \cite{ES2}, where various subcases were considered. There the authors prove, among other results, that for a smooth connected projective variety over an algebraically closed field of positive characteristic, the triviality of the \'etale fundamental group implies that any convergent isocrystal, which is filtered so that the associated graded is a sum of rank 1 convergent isocrystals, is constant
 \cite[Theorem 0.1]{ES1}. In particular, it is also proven that rank 1 isocrystals (not necessarily convergent) are also trivial in this case  \cite[Theorem 1.2]{ES2}.
In the present work we want to extend the aforementioned results of \cite{ES1} to the case of a rank one log extendable  isocrystal on  a non-proper variety with trivial \'etale fundamental group. Actually we just need to assume that its abelianized tame fundamental group is trivial. More precisely, we prove the following:

\begin{theorem}\label{main}
  Let $U$ be a smooth connected variety over an algebraically closed field $k$ of positive characteristic, that admits a good compactification $X$, where  $X-U =\cup_{i\in I} Z_i=:Z$ is a simple strict normal crossings divisor. If $\pi_1^{\op{tame,ab}}(U)=1$ and $\mathcal L$ is a rank one isocrystal on $U$ that extends to a log isocrystal on $X$, equipped with the log structure coming from $Z$, then $\mathcal L$ is trivial.
\end{theorem}

As a corollary we obtain also, with the previous notations and assumptions:
\begin{corollary}
  Unipotent log extendable isocrystals on $U$ are constant.
\end{corollary}
For the proof of our main result we adapt  mainly the strategy  \cite[Proposition 3.6]{ES1} and \cite[Sections 2.1 and  4]{ES2}.

In more detail, starting with a rank one isocrystal $\mathcal L$ on $U$, as in the assumptions of the theorem, we denote by $\cl{L}^{\llog}$ the log isocrystal on $X$ extending $\cl{L}$, by $L^{\llog}$ a locally free lattice of it and by $L$ the restriction of the latter to $U$. We first work on the value of $\cl{L}^{\llog}$ on $X$. It can be seen as a log crystal on $X/k$ (we omit here the reference to the log structures) and can be seen as a module on the log scheme $X$ with integrable quasi nilpotent log connection, which we denote by $(L^{\llog}_X,\nabla^{\llog}_X)$.

In Section \ref{residuesection} we adapt the discussion  of \cite[Section 6]{ABV} to our situation and obtain a residue exact sequence:
\begin{equation}\label{residueseq2}
H^1_{crys}(X/W,\mathcal O^*_{X/W})\to H^1((X,M)/W,\mathcal O^*_{X/W})_{\op{crys}}^{\llog} \to \oplus W(k)[Z_i] \to  H^2(X/W,\mathcal O^*_{X/W}).
\end{equation}
One also derives from (\ref{residueseq2}) that there is an injection \begin{equation} H^1((X,M)/W,\mathcal O^*)_{\op{crys}}^{\llog}\hookrightarrow  \oplus W(k)[Z_i].\end{equation}

We then prove that the log crystalline Chern class $c_1^{\op{logcrys}}(L^{\llog}_X)$ is equal to zero and using this we conclude in Lemma \ref{formalsum}, that the line bundle $L_X$ to some power $N$ is then of the form $\mathcal O_X(\sum a_i Z_i)$ with $\mathbb Q$-coefficients. To simplify the notations set $E^{\llog}:=(L^{\llog})^{\otimes N}$.

By (\ref{residueseq2}) we can then conclude that the map $H^1((X,M)/W,\mathcal O^*)_{\op{crys}}^{\llog}\to H^1(X,\mathcal O^*_X)$ is injective, hence that the log crystal $(E^{\llog}_X,\nabla^{\llog}_X)$ is actually ``controlled'' by the value of the sheaf $E^{\llog}_X$.

Using this and the previous observations we get that the restriction of $(E^{\llog}_X,\nabla^{\llog}_X)$ to $U$ is the trivial module with the trivial connection. It then follows that we can assume $(E^{\llog}_X,\nabla^{\llog}_X)$ to be isomorphic to $(\cl{O}_X^{\llog},d^{\llog})$.

In Theorem \ref{triviality} we then use a deformation argument (adapted from \cite{ES1}) to prove that $L^{\otimes n}$ is the trivial crystal and conclude that $L$ and therefore also $\mathcal L$ are trivial, using Lemmas \ref{Kummer} and \ref{pPower}.

As a corollary of this theorem and using the observations made in Section \ref{residuesection}, we also obtain in Corollary \ref{exttrivial} that extensions in $\op{I}_{\op{crys}}(X/W)^{\llog}$ of the trivial log isocrystal by itself are constant and so isocrystals on the open variety whose extensions are of this form, are also constant. Therefore also log extendable unipotent isocrystals on $U$, i.e. isocrystals that admit a filtration with the property that the associated quotients are successive extensions of the trivial isocrystal, are also constant.

Finally, we would like to point out the steps of the proof where the triviality of $\pi_1^{\op{tame,ab}}(U)$. From (\ref{tamedecomp}) we obtain that then also $\pi^{\op{ab}}(X)=1$. In  Lemma \ref{formalsum} we use that the maximal pro-$l$-quotient $\pi_1^{\op{ab},(l)}(U)$ is trivial for $l\neq p$ to obtain that $H^0(U,\mathcal O_U^*)=k^*$, while $\pi^{\op{ab}}(X)=1$ also implies $\operatorname{NS}(X)=\operatorname{Pic}(X)$. For the observations of Section \ref{residuesection}, the triviality of   $\pi_1^{\op{tame,ab}}(U)$ also suffices. Indeed, there we use that $H^1_{\op{crys}}(X/W)=H^1_{\op{crys}}(X/W,\mathcal O_X^*) =0 $, from \cite[Theorem 0.1.1 and Proposition 2.9]{ES1}  and \cite[Theorem 5.1]{ES2}. The triviality of the abelianized tame fundamental group is indeed enough to derive these equalities. The same relations are also mainly used in Theorem \ref{triviality} together with the fact that $\operatorname{NS}(X)=\operatorname{Pic}(X)$. To conclude Section \ref{residuesection} and for Lemma \ref{Kummer} we use that Kummer coverings of $U$ are trivial, thus we have to assume that $\pi_1^{\op{tame,ab}}(U)=1$, since $k$ is algebraically closed.

 \medskip 
 
 {\it Acknowledgements: }This work is part of my PhD project at the Freie Universit\"at Berlin. I would like to thank my advisor Prof. H\'el\`ene Esnault for her guidance and constant support. I would also like to thank Prof. Tomoyuki Abe and Prof. Atsushi Shiho for their most valuable advice and for the very helpful discussions we had during their visits in Berlin. In this second version of the article, I used some very useful suggestions of a first reviewer of the article to improve the presentation.

\section{Preliminaries and notation}
Let  $U$ be a  smooth connected variety over an algebraically closed field $k$ with the property that $\pi_1^{\op{tame,ab}}(U)=1$ and such that it admits a good compactification $ X$, where  $X-U =\cup_{i\in I} Z_i=:Z$ is a simple strict normal crossings divisor. Denote by $W$ the Witt ring of $k$ and by $K$ the fraction field of $W$.

Note the following fact regarding the abelianization of the tame fundamental group of $U$, see \cite{Sza}:
\begin{equation}
\pi_1^{\op{tame,ab}}(U)  = \pi_1^{\op{ab}}(U)(p')\oplus \pi_1^{\op{ab}}(X)(p)\label{tamedecomp},
\end{equation}
where $(p)$ and $(p')$ denote the maximal pro-$p$ and prime-to-$p$ quotients respectively. % This relation follows from the decomposition of $\pi_1^{tame,ab}(U)$ to its pro-$p$ and prime-to-$p$ quotients and the fact that a p-cover of U lifts to a p- cover of X.

From this decomposition we can deduce that if $\pi_1^{\op{tame,ab}}(U) = 1 $ then in particular $\pi_1^{\op{ab}}(U)(p')= \pi_1^{\op{ab}}(X)(p)= 1$. Since we also have a surjection $\pi_1^{\op{ab}}(U)(p')\to \pi_1^{\op{ab}}(X)(p')$, we see that $\pi_1^{\op{ab}}(X)(p')$ must be also trivial, which means that $\pi_1^{\op{ab}}(X)=1$. 

\subsection*{Log structures}

We gather here some definitions and facts that will be used in later sections:

We endow $X$ with the log structure associated to the strict normal crossings divisor $Z$ and denote this log scheme by $(X,M_Z)$ or $(X,Z)$: if we denote the open immersion $X\setminus Z\hookrightarrow X$ by $j$, the inclusion
\begin{equation}
  M_Z:= \mathcal O_X\cap j_*\mathcal O^* _{X-Z}\hookrightarrow \mathcal O_X
\end{equation}
defines a log structure on $X$.

\begin{definition} [Trivial log structure] \cite[Example 2.4.1]{ShI}. 
  
Let $Y$ be a scheme or a $p$-adic formal scheme. The pair ($\mathcal O_Y^*, \mathcal O_Y^*\hookrightarrow \mathcal O_Y$) is a log structure on $Y$, called the trivial log structure. This log structure is associated to the pre-log structure $\{1\}_Y\to \mathcal O_Y.$
\end{definition}

Since $Z$ is a normal crossings divisor, we know by \cite[Example 2.4.4]{ShI} that $M_Z$ is an fs, hence also fine, log structure \cite[Definition 2.1.5]{ShI} and the morphism
$(X,M_Z) \to (\operatorname{Spec}k, \text{triv.log str.})$ is log smooth.

\begin{definition}{\cite{ShII},\cite{HyKa}}
Assume that  we have a morphism of fine log schemes
\begin{equation*}
	(X,M)\to(\op{Spf}V,N)
\end{equation*}
with $X/\operatorname{Spf} V$ of finite type and $V$ a totally ramified finite extension of $W$.

A good embedding system is a diagram
\begin{equation}
\begin{tikzcd}
(X^\bullet, M^{\bullet}) \arrow{r}\arrow{d} & (P^{\bullet},N^\bullet)\arrow{d}\\
(X,M)\arrow{r} & (\op{Spf}V,N)
\end{tikzcd}
\end{equation}
with the following properties:

\begin{enumerate}
	\item $X^i \to X$ is a  hyper-covering for the \'etale topology  and $M^i (i\ge 0)$ is the inverse image of $M$ on $X^i$, for all $i$.
	\item  each $(X^i, M^i)$ is a simplicial fine log scheme, of
	 Zariski type and of finite type over $k$.
	\item each $(P^i,N^i)$ is a  simplicial fine formal log $V$ scheme, formally log smooth and of Zariski type.
	\item $(X^i,M^i)\to (P^i,N^i)$ is a closed immersion for all $i$.

\end{enumerate}
\end{definition}

\begin{remark}
  By  \cite[Propostion 2.2.11]{ShII}, given morphisms of fine log schemes
	$(X,M)\xrightarrow{f} (\operatorname{Spec} k)\xrightarrow{i}(\op{Spf}V,N)$
	where $f$ is of finite type, $i$ is the canonical exact closed immersion and assuming that $(\op{Spf}V,N)$ admits a chart,  there exists at least one good embedding system of $(X,M) $ over $(\op{Spf}V, N)$.
\end{remark}

\subsection*{Lattices}

\begin{definition}
  We call an isocrystal $\cl{L}$ on an open smooth scheme $U$ log extendable if there is a log isocrystal $\cl{L}^{\llog}$ on $X$ such that the restriction of $\cl{L}^{\llog}$ on $U$ is equal to $\cl{L}$.
\end{definition}

 \begin{definition} A lattice for an isocrystal $ \cl{L}$ on a scheme $U$ is a $p$-torsion free crystal $L\in \op{Crys}(U/W)$ such that $\cl{L}\simeq L\otimes \mathbb Q$.\end{definition}
 
 For a given isocrystal we can find more than one lattices and it is not in generally true that an isocrystal always admits a locally free lattice, i.e. a lattice $L$ whose value on $U$ is a locally free coherent sheaf.
 
 However, when the isocrystal is of rank 1, it admits a locally free lattice by \cite[Proposition 2.10]{ES1}.
 
 The same is indeed true for a rank 1 log isocrystal on a proper variety $X$:
 \begin{lemma}
 	A rank 1 log isocrystal on a proper variety $X$ admits a locally free lattice. 	
      \end{lemma}
      
 \begin{proof}
 	The argument is the same as in  \cite[Proposition 2.10]{ES1} in the local case. We can glue the local lattices as in  \cite[Lemma 2.11]{ES1}, because $H^0(X,\cl{O}_X)^{\llog}_{\op{crys}} = H^0(X,\cl{O}_X)_{\op{crys}}$ by \cite[Section 6]{ABV}.
      \end{proof}

\section{Residue exact sequence}\label{residuesection}

We follow \cite[Section 6]{ABV}.
Let $\{U_i\hookrightarrow V_i\}$ be a covering in the log crystalline site. We set $U_{ij}=U_i\cap U_j$ and denote by $V_{ij}$ the log PD envelope of $U_{ij}$ in $V_i$ (or $V_j$) and do the same for every tuple of indices $(i_0,i_1,\cdots,i_n)$. 
Define $Q^r_{i_0,\dots,i_n}:=\omega^{r,log}_{V_{i_0,\dots,i_n}}/ \omega^{r}_{V_{i_0,\dots,i_n}}$. We then have the following maps between the total complexes of the corresponding \v{C}ech complexes, with exact columns:

\begin{equation}
\begin{tikzcd}
0\arrow{r}& \oplus \mathcal O^*_{V_i}\arrow{d}\arrow{r} & \oplus_{i,j} \mathcal O^*_{V_{ij}} \oplus_i \omega^1_{V_{i}} \arrow{d}\arrow{r} & \oplus_{i,j,k} \mathcal O^* _{V_{ijk}} \oplus_{i,j} \omega^1_{V_{ij}} \oplus_{i,j} \omega^2_{V_{ij}}\arrow{r}\arrow{d} & \dots\\
0\arrow{r}& \oplus \mathcal O^*_{V_i}\arrow{d}\arrow{r} & \oplus_{i,j} \mathcal O^*_{V_{ij}} \oplus_i \omega^{1,\llog}_{V_{i}} \arrow{d}\arrow{r} & \oplus_{i,j,k} \mathcal O^* _{V_{ijk}} \oplus_{i,j} \omega^{1,\llog}_{V_{ij}} \oplus_{i,j} \omega^{2,\llog}_{V_{ij}}\arrow{r}\arrow{d} & \dots\\
&0\arrow{r} & \oplus Q_i^1\arrow{d} \arrow{r} & \oplus_{i,j} Q^1_{ij} \oplus_{i} Q_i^2\arrow{r} \arrow{d}& \dots\\
&&0&0 
\end{tikzcd}
\end{equation}
and note  that the third complex is the total complex of:
\begin{equation}\label{qcomp}
\begin{tikzcd}
0\arrow{r}\arrow{d}  & \oplus Q_i^1 \arrow{r}\arrow{d} & \oplus Q_{ij}^1 \arrow{r}\arrow{d}&\dots\\
0\arrow{r}\arrow{d}  & \oplus Q_i^2 \arrow{r}\arrow{d} & \oplus Q_{ij}^2 \arrow{r}\arrow{d}&\dots\\
0\arrow{r}\arrow{d}  & \oplus Q_i^3\arrow{r}\arrow{d} & \oplus Q_{ij}^3 \arrow{r}\arrow{d}&\dots\\
0& \dots & \dots &\dots
\end{tikzcd}
\end{equation}

Taking cohomology, we obtain maps (starting from $H^1$)
\begin{equation}
H^1_{\op{crys}}(X/W,\mathcal O^*_{X/W})\to H^1((X,M)/W,\mathcal O^*_{X/W})_{\op{crys}}^{\llog} \to H^1(Tot(Q_{\bullet,\bullet})) \to H^2(X/W,\mathcal O^*_{X/W})
\end{equation}
and computing the cohomology of the complex (\ref{qcomp}) with spectral sequences, using the fact that $H^1(Tot(Q_{\bullet,\bullet})) = hor E\infty _{10} \oplus hor E^\infty _{01}$, we obtain the exact sequence
\begin{multline}
H^1_{\op{crys}}(X/W,\mathcal O^*_{X/W})\to H^1((X,M)/W,\mathcal O^*_{X/W})_{\op{crys}}^{\llog} \to  \\
\to \op{Ker}(H^0(Q^1_\bullet) \to H^0(Q^2_\bullet) ) \to H^2(X/W,\mathcal O^*_{X/W}) .
\end{multline}
By the computation in \cite[Proposition 6.4]{ABV}  we have that
\begin{equation} \op{Ker}(H^0(Q^1_\bullet) \to H^0(Q^2_\bullet) )= \op{Div}_Z(X) \otimes W(k),\text{ or } \oplus W(k)[Z_i], \end{equation}
so we obtain:
\begin{equation}\label{redidueSeq}
H^1_{\op{crys}}(X/W,\mathcal O^*_{X/W})\to H^1((X,M)/W,\mathcal O^*_{X/W})_{\op{crys}}^{\llog} \to \oplus W(k)[Z_i] 
\to  H^2(X/W,\mathcal O^*_{X/W}) .
\end{equation}
Moreover, in \cite{ABV} we find the exact sequence
\begin{equation}\label{abvexact}
0\to H^1_{\op{crys}}(X/W)\to H^1((X,M)/W)_{\op{crys}}^{\llog} \to \oplus W(k)[Z_i] \to H_{\op{crys}}^2(X/W)
\end{equation}
with the last map being the crystalline Chern class map. By \cite[Proposition 2.9(2)]{ES1} and  \cite[Theorem 4.3.1]{Esfund} we have that $\pi^{\op{ab}}_1(X)=1$ implies $H^1_{\op{crys}}(X/W)=0$. Hence, the group $H^1(X/W)_{\op{crys}}^{\llog}$ coincides in this case with the kernel of the crystalline Chern class map, which by \cite[Proposition 6.4]{ABV} is equal to  $\op{Div}_Z^0(X)\otimes W(k)$ and this is equal to zero because $\operatorname{Pic^0}(X)=0$.
We therefore obtain an injection $ \oplus W(k)[Z_i] \hookrightarrow H_{\op{crys}}^2(X/W)$. 

For $X$  proper we have that $H^1_{\op{crys}}(X/W,\mathcal O^*_X)=0$, by \cite[Theorem 0.1.1]{ES1}  and  \cite[Theorem 4.3.1]{Esfund}. Therefore (\ref{redidueSeq}) gives us an injection
\begin{equation}
H^1((X,M)/W,\mathcal O^*)_{\op{crys}}^{\llog}\hookrightarrow  \oplus W(k)[Z_i]
\end{equation}
and combining the two exact sequences we obtain:
\begin{equation}\label{inj1}
H^1((X,M)/W,\mathcal O^*)_{\op{crys}}^{\llog}\hookrightarrow  \oplus W(k)[Z_i] \hookrightarrow   H_{\op{crys}}^2(X/W).
\end{equation}

From the exact sequences above and going through the definition of $c_1^{\op{crys}}$ and $c_1^{\op{logcrys}}$ (see also Lemma \ref{chernlemma} for the definition) we observe, as already remarked in (Ibid.), that the following diagram commutes:
\begin{equation}
  \begin{tikzcd}
    H^1((X,M)/W,\mathcal O^*)_{\op{crys}}^{\llog}\arrow[hook]{r}\arrow{dr} & \oplus W(k)[Z_i]\arrow[hook]{r}{c_1^{\op{crys}}} & H^2(X/W)_{\op{crys}}\\
   & H^1(X,\mathcal O^*_X)\arrow{ur}{c_1^{\op{crys}}}&
  \end{tikzcd}\hfill .
\end{equation}
From this we conclude that under our assumptions the projection map
\begin{equation}\label{inj2} H^1((X,M)/W,\mathcal O^*)_{\op{crys}}^{\llog}\to  H^1(X,\mathcal O^*_X), \qquad (L,\nabla)\mapsto L_X\end{equation}
is in fact injective.

\begin{remark}\label{observation}
	We can show  that $H^1_{rig}(U/K)=0$ under our assumptions. For a proper scheme this is known.
	Observing (\ref{abvexact}) we conclude as before that 
	$H^1(X/W)_{\op{crys}}^{\llog}=0$. Hence $H^1_{rig}(U/K)\simeq H^1(X/W)_{\op{crys}}^{\llog}\otimes \mathbb Q=0$.
      \end{remark}

\section{The value of the log extension on $X$}
Starting with $\cl{L}\in \op{I}_{\op{crys}}(U/K)$, we take an extension $\cl{L}^{\llog}\in \op{I}_{\op{logcrys}}(X/K)$ and, as discussed above, a locally free log lattice $L$ of $\cl{L}^{\llog}$. We denote its restriction to $U$ by $L$.
In the following we denote by $L^{\llog}_X \in \op{Coh}(\cl{O}_X)$ the value of the log crystal on $X$.

\begin{lemma}\label{chernlemma}
With the above notations we have $c_1^{\op{logcrys}}(L^{\llog}_X)=0$.
\end{lemma}
\begin{proof}

  One defines the crystalline Chern class of a locally free sheaf in analogy with \cite{chern}: let $\cl{J}^{\llog}_{X/W}$ be the  log PD ideal defined by the short exact sequence
  \begin{equation}\label{firstlogchern}
    0\to  \cl{J}^{\llog}_{X/W} \to \cl{O}_{X/W}^{\llog} \to i_*\cl{O}_{X} \to 0
  \end{equation}
  with $i:(X)_{\op{Zar}} \to (X/W)^{\llog}_{\op{crys}}$. Consider also
  \begin{equation}\label{secondlogchern}
    1\to 1+ \cl{J}^{\llog}_{X/W} \to \cl{O}_{X/W}^{*,\llog} \to i_*\cl{O}^*_{X} \to 1.
    \end{equation}

    The first log crystalline Chern class is defined by composing the map $H^1(X,\cl{O}^*_X) \to H^2(X/W, \cl{J}^{\llog}_{X/W} )^{\llog}_{\op{crys}}$ followed by the logarithm map
    \begin{equation}
    1 + \cl{J}^{\llog}_{X/W} \to \cl{J}^{\llog}_{X/W},\text{ } 1+x\mapsto \op{log}(1+x).  \end{equation} So we obtain $\op{c}_1^{\op{logcrys}} : H^1(X,\cl{O}^*_X) \to H^2(X/W, \cl{O}^{*,\op{log}}_{X/W})^{\llog}_{\op{crys}}.$

    Taking cohomology of \ref{firstlogchern} we get
      \begin{multline}
      0\to H^0(X/W, \cl{J}^{\op{log}}_{X/W})^{\op{log}}_{\op{crys}} \to
      H^0(X/W, \cl{O}^{\op{log}}_{X/W})^{\llog}_{crys}\to\\
      \to H^0(X,\cl{O}_X)\to H^1(X/W, \cl{J}^{\op{log}}_{X/W})^{\llog}_{\op{crys}}
      \to
      H^1(X/W, \cl{O}^{\op{log}}_{X/W})^{\llog}_{\op{crys}}.
    \end{multline}
 By Remark \ref{observation} we have however that $ H^1(X/W, \cl{O}^{\op{log}}_{X/W})^{\llog}_{\op{crys}}=0$.

    Moreover, since $X$ is geometrically connected we have that $ H^0(X/W, \cl{O}^{\op{log}}_{X/W})^{\llog}_{\op{crys}}\simeq W$ which surjects to $ H^0(X,\cl{O}_X)\simeq k$. Hence we see that $H^1(X/W, \cl{J}^{\op{log}}_{X/W})^{\llog}_{\op{crys}}\simeq H^1(X/W,1+ \cl{J}^{\op{log}}_{X/W})^{\llog}_{\op{crys}} =0$.

    Therefore we obtain an exact sequence
      \begin{equation}
      0\to  H^1(X/W, \cl{O}^{*,\llog}_{X/W})^{\llog}_{\op{crys}} \to H^1(X,\cl{O}_X)\to H^2(X/W, \cl{J}^{\op{log}}_{X/W})^{\llog}_{\op{crys}}.
    \end{equation}
    By applying the map
    \begin{equation}
      H^2(X/W, \cl{J}^{\op{log}}_{X/W})^{\llog}_{\op{crys}}\to H^2(X/W, 1+ \cl{J}^{\op{log}}_{X/W})^{\llog}_{\op{crys}}\to  H^2(X/W, \cl{O}^{*,\op{log}}_{X/W})^{\llog}_{\op{crys}}
    \end{equation}
    we get
    \begin{equation}
 0\to  H^1(X/W, \cl{O}^{*\llog}_{X/W})^{\llog}_{\op{crys}} \to H^1(X,\cl{O}_X)\xrightarrow{c_1^{\llogcr}} H^2(X/W, \cl{O}^{*,\op{log}}_{X/W})^{\llog}_{\op{crys}}.
      \end{equation}
   Therefore we have that the value of a rank 1 crystal on $X$ is sent to zero by $c_1^{\op{logcrys}}$.
\end{proof}

\begin{lemma}\label{formalsum}
  The line bundle $(L_X^{\llog})^{\otimes N}$ for some natural number $N$ can be written as $\mathcal O_X(\sum a_i Z_i)$, with $a_i\in\mathbb Q$.
\end{lemma}
\begin{proof}
We have an exact sequence:
\begin{equation}
0\to H^0(X,\mathcal O_X^*)\to H^0(U,\mathcal O_U^*)\to\oplus\mathbb Z[Z_i]\to \operatorname{Pic}(X)\to \operatorname{Pic}(U)\to 0
\end{equation}
By \cite[Proposition 3.5]{evidence} we have that  $\pi_1^{\op{ab}}(X)=\pi_1^{\op{ab}}(U)=1$ implies $H^0(U,\mathcal O^*_U)=k^*$ and thus $H^0(X, \mathcal O^*_{X})$ is isomorphic  to $H^0(U, \mathcal O^*_U)$.

Moreover, it also implies that $\operatorname{Pic}(X)=\operatorname{NS}(X)$ and $\operatorname{Pic}(U)=\operatorname{NS}(U)$, and the exact sequence becomes
\begin{equation}\label{Pic}
0\to\oplus\mathbb Z[Z_i]\to \operatorname{NS}(X)\to \operatorname{NS}(U)\to 0.
\end{equation}
The log crystalline Chern class can be seen as \begin{equation}c_1^{\op{logcrys}}:H^1(X,\mathcal O_X^*)\to H^2(X/W)^{\llog}_{\op{crys}} \otimes \mathbb Q\simeq H^2_{rig}(U/K). \end{equation}

We observe that the following diagram commutes:
 \begin{equation}\label{cryslog}
  \begin{tikzcd}
   H^1(X,\mathcal O_X^*)\arrow{r}{c_1^{\op{logcrys}}}\arrow{d}{c_1^{\op{crys}}} & H^2(X/W)^{\llog}_{\op{crys}}\\
    H^2(X/W)_{\op{crys}}\arrow{ru}{f^*}
  \end{tikzcd}\hfill ,
\end{equation}
where we denote by $f^*: H^2(X/W)_{\op{crys}}\to H^2(X/W)^{\llog}_{\op{crys}}$ the homomorphism  which is obtained by the exact map of topoi $f:(X/W_n)^{\llog}_{\op{crys}}\to (X/W_n)_{\op{crys}}$ for all $n$
and  is defined as  $f^*(E)((U,M_U),(T,M_T)):=E(U\subset T)$.
As stated in \cite{ABV} this map is exact and commutes with global sections. 
Applying this to the sequence
\begin{equation}
  1\to 1+\mathcal J_X\to \mathcal O^*_{X/W_n}\to \mathcal O^*_X\to 1
\end{equation}
we get
\begin{equation}
  \begin{tikzcd}
    1\arrow{r} & f^*( 1+\mathcal J_X)\arrow{r} \arrow{d} & f^*( \mathcal O^*_{X/W_n})\arrow{r}\arrow{d} & f^*( \mathcal O^*_X)\arrow{r}\arrow{d} &  1\\
  1\arrow{r} &  1+\mathcal J_X^{\llog} \arrow{r}&  \mathcal O^{*,\llog}_{X/W_n}\arrow{r}&  \mathcal O^*_X \arrow{r} &  1
\end{tikzcd}
\end{equation}
with exact rows and commutative squares. By taking cohomology of the first row and composing with the log map, one obtains $f^*(c_1^{\op{crys}})$ whereas by taking cohomology of the second, one obtains $c_1^{\op{logcrys}}$.

Combining the previous observation with \cite[Proposition 3.6.1]{Petr} we obtain the following commutative diagrams
\begin{equation}
\begin{tikzcd}
  H^2_{rig}(X/K)\arrow{r} & H^2(X/W)_{\op{crys}}\otimes \mathbb Q \arrow{d}\\
  H^1(X,\mathcal O^*)\arrow{u}{c_1^{rig}} \arrow{ur}{c_1^{\op{crys}}}\arrow{r}{c_1^{\op{logcrys}}} & H^2(X/W)^{\llog}_{\op{crys}}\otimes \mathbb Q
 \end{tikzcd}\hfill .
\end{equation}

Moreover, by \cite[Remark II 6.8.4]{ill}, and using the fact that  $\pi_1^{\op{ab}}(X)=1$ implies also that $\operatorname{Pic}(X)=\operatorname{NS}(X)$,  we know that $H^1(X,\mathcal O^*_X)\otimes K \hookrightarrow H^2_{rig}(X/K)$ via the first rigid Chern class.

The line bundle $L^{\llog}_X$ seen in $\operatorname{NS}(X)\otimes\mathbb Q$ (i.e. $L^{\llog}_X$ modulo torsion) is sent to zero by $c_1^{\op{logcrys}}$, by Lemma \ref{chernlemma}.

So we have
\begin{equation*}
  \begin{tikzcd}
    &\left<K[Z_i]\right>_{\mathbb Q}\arrow[hookrightarrow]{d}\\
   \operatorname{NS}(X)\otimes \mathbb Q\hookrightarrow \operatorname{NS}(X)\otimes K \arrow[hookrightarrow]{r}{c_1^{rig}} \arrow{dr}{j^*c_1^{rig}}& H^2_{rig}(X/K) \arrow{d}{j}\\
    & H^2_{rig}(U/K)
 \end{tikzcd}
 \end{equation*}
which means that $L^{\llog}_X$ is contained in the kernel of the map $ H^2_{rig}(X/K)\to H^2_{rig}(U/K)$, which is $\left<K[Z_i]\right>_{\mathbb Q}$. Hence, $L^{\llog}_X\in \left<K[Z_i]\right>_{\mathbb Q} \cap \operatorname{NS}(X)\otimes \mathbb Q=\oplus \mathbb Q[Z_i]$.

This in turn means that $L^{\llog}_X$ can be written as $\mathcal O_X(\sum a_i Z_i) \otimes T$, with $T$ a torsion bundle, 
and therefore some power of it can be written as a formal sum of $Z_i$ with $\mathbb Q$-coefficients.
\end{proof}

Of course we could increase the value of $N$ and obtain integral coefficients in the above decomposition. As we see in the next section it is actually enough to raise the bundle $(L^{\llog}_X)^{\otimes N}$ to a power that is prime to $p$ and obtain the same result.

To simplify the notation set $E^{\llog}:=(L^{\llog})^{\otimes N}$.

\begin{remark}
Recall that under our assumptions Lemma \ref{formalsum} implies that the sheaf $E^{\llog}:=(L^{\llog})^{\otimes N}$ for some $N$ lies in $\oplus \mathbb Q[Z_i]$, whereas from (\ref{inj1}) it can also be written as a sum with $W(k)$ coefficients. Hence, $E^{\llog}_X$ lies actually in $\oplus \mathbb{Z}_{(p)}[Z_i]$.
Therefore there is $m \in \mathbb Z\setminus p\mathbb Z$ with the property that  $E_X^{\otimes m}\in  \oplus \mathbb Z[Z_i]$ is such that $E_U^{\otimes m} =  \mathcal O_U$ and $(m,p)=1$. This in turn defines a Kummer cover of $U$. However, by assumption we have that $\pi_1^{\op{tame,ab}}(U)=1$, so $m$ can be chosen as $1$ and $E_X$ is in fact in $\oplus \mathbb Z[Z_i]$ and has the property that $E_U$ is trivial (see Lemma \ref{Kummer} for  detailed argument).

On the other hand, from (\ref{inj2}) we see that the connection $\nabla^{\llog}_X$, as well as its restriction to $U$ can be assumed to be trivial.

We therefore have a module with  log connection on $X$ restricting to $(\cl{O}_U,d_U)$ on $U$ and can take an extension of it that is isomorphic to $(\cl{O}^{\llog}_X,d_X^{\llog})$.
\end{remark}
  
\section{Triviality of isocrystal}
 \begin{theorem}\label{triviality}
	Let $E$ a locally free lattice of a rank 1 log isocrystal $\mathcal E$ on $(X,Z)$ and such that its value on $(X,Z)$ is trivial. Then $E$ is the trivial log crystal. 
\end{theorem}
\begin{proof}
	
	\textbf{Step 1}:
	We have the following maps of fine log schemes:
	\begin{equation*}
	(X, Z)\rightarrow \underline S \rightarrow \underline W(k)
	\end{equation*} 
	where $\underline S$ is the log scheme defined by $S=\operatorname{Spec} k$ with the log structure associated to $\mathbb N\to k;1\mapsto 0$ and $\underline W(k)$ is the log scheme defined by $\operatorname{Spec } W(k)$ with the log structure $\mathbb N\to W(k);1\mapsto 0$.

 There is at least one good embedding system for $(X,Z)/(\operatorname{Spec} k,N)$, so we assume from now on that there is an embedding $(X,Z)\hookrightarrow (P,L)$ to a smooth formal scheme over $(\operatorname{Spf} W,N)$.
	
	Denote by $D$ the completed log PD envelope of $( X\hookrightarrow P$). It is the usual PD envelope equipped with the inverse image of the log structure. Denote by $ D_n$ the corresponding PD envelopes for the immersions over $P/\operatorname{Spf} W_n$. 
\begin{equation}	
	\begin{tikzcd}
		D_n\arrow{dr}\\
		 X \arrow{u}\arrow{r}\arrow{d}&  P_{W_n}\arrow{dl}\\
		W_n
	\end{tikzcd}\hfill .
\end{equation}	
	By  \cite[Theorem 6.2]{Kat} the category  $C_{\op{crys}}( X/W_n)^{\llog}$ of log crystals on $ X/W_n$ is equivalent to the category of $\mathcal O_{{\mathcal D_n}}$-modules with integrable log quasi nilpotent connection.
	
Let $E$ be a log crystal such that the restriction $E_n$ of $E_{n+1}$ to $C_{\op{crys}}(X/W_n)^{\llog}$ is the trivial log crystal.  We can define, as in \cite[Proposition 3.6]{ES1}, $\mathcal D$ to be the set of pairs $(G,\phi$) such that  $G\in C_{\op{crys}}(X/W_{n+1})^{\llog}$ with $\phi: \mathcal O_X \simeq \overline G\in C_{\op{crys}}(X/W_{n})^{\llog}$.
We claim that in this case we have an isomorphism  $e:\mathcal D\simeq H^1( X/W_1)^{\llog}_{\op{crys}}$.
By the equivalence of log crystals and modules with log integrable quasi nilpotent connection, we can identify $E_{n+1}$ with a  $\mathcal O_{\mathcal D_{n+1}}$ module $M_{n+1}$ with a log connection $\nabla_{n+1}$. By assumption, there is an isomorphism
 \begin{equation*} \phi:(\mathcal O_{D_n}, d) \to (M_n,\nabla_n),\end{equation*} $ (M_n,\nabla_n)$ being the restriction to $\op{MIC}^{\llog}(D_n)^{qn}\simeq C_{\op{crys}}(X/W_n)^{\llog}$.

Take an open affine cover $U=\{U_a\}$ of $D_{n+1}$ and an isomorphism $\psi_a:\mathcal O_{U_a}\to M_{n+1}|_{U_a}$ lifting $\phi|_{D_n\times U_a}$.
Then $\psi_a^*(\nabla_{n+1})$ defines a connection on $\mathcal O_{Ua}$ that is given as $d + p^ns_a$ with $s_a \in \Gamma(U_a,\omega^1_{\llog,D_1})$. 
	% where $\omega^1_{Y/S}=\Omega^1(log(M_Y/M_S))$ for two log schemes $Y,S$.
On the intersection $U_a\cap U_b=:U_{ab}$ the glueing $(\psi_a|_{U_{ab}})^{-1} \circ (\psi_b|_{U_{ab}})$ is given by $1+p^n z_{ab}$ with $z_{ab}\in \Gamma(U_{ab},\mathcal O_{D_1})$.

Because of the integrability of the connection and the compatibility of the connection with the glueing, we see as in \cite{ES1}, that $(\{s_a\},\{z_{ab}\})$ actually defines a 1-cocycle in 
$\op{Tot}\Gamma(U,\omega^\bullet_{\llog,D_1})$ and this defines a class  in the cohomology $H^1(\op{Tot}\Gamma(U,\omega^\bullet_{\llog,D_1}))=H^1( X/W_1)^{\llog}_{\op{crys}}$.
	
\textbf{Step 2}: By the previous step we get:
	\begin{equation*}
	\op{Ker}(\op{MIC}^{\llog}(D_{n+1})\to \op{MIC}^{\llog}(D_n))\simeq H^1(X,\omega^\bullet_{\llog,D_1}).
	\end{equation*}
	We know the following from \cite[Theorem 2.4.4 and Corollary  2.3.9]{ShII}, denoting by $K_{ X/W }$ the convergent isocrystal defined by $T\mapsto K\otimes_W \Gamma(T,\mathcal O_T)$: for all $i\in \mathbb N$
\begin{equation}
         H^i(( X/W)^{\llog}_{\op{crys}}, K\otimes \mathcal  O_{X/W})\simeq H^i(( X/W)^{\llog}_{conv},K_{ X/W })
\end{equation}	
and
\begin{equation}
  H^i(( X/W)^{\llog}_{conv},K_{ X/W }) \simeq H^i_{rig}(U/K).
  \end{equation}
In  \cite[page 136 ]{ShII}  Shiho defines for $\mathcal E=K\otimes \mathcal F$ in $\mathcal I_{\op{crys}}(X/W)^{\llog}$:
	
\begin{equation*}  H^1(( X/W)^{\llog}_{\op{crys}},\mathcal E)=\mathbb Q\otimes_\mathbb Z H^1(( X/W)^{\llog}_{\op{crys}},\mathcal F). \end{equation*}
	
Combining this with Remark \ref{observation}
 \begin{equation*}
  H^1(( X/W)^{\llog}_{\op{crys}}, K\otimes \mathcal O_{ X/W})=\mathbb Q \otimes H^1( X/W)^{\llog}_{\op{crys}}, \mathcal  O_{ X/W})= H^1_{rig}(U/K)=0.
   \end{equation*}
	
From \cite[Theorem B' and Section 2.3]{ABV}, we know that $H^1(X/W)^{\llog}_{\op{crys}}$ is a free $W$-module, hence torsion free and therefore $ H^1( X/W)^{\llog}_{\op{crys}}=\underset{n}{\varprojlim} H^1( X/W_n)^{\llog}_{\op{crys}}=0$. 
Hence there is a natural number $N$ large enough such that for all $n\ge N$,
the maps $H^1( X/W_n)^{\llog}_{\op{crys}}\to H^1( X/W_1)^{\llog}_{\op{crys}}$ are zero.
	
\begin{claim}: There is a natural number $d\ge 0$ s.t. $(F^d)^*H^1( X/W_1)^{\llog}_{\op{crys}}=0$.

From \cite[Corollary on page 143]{mum} we have the decomposition of 
$H^1(X/W_1)^{\llog}_{\op{crys}}$ as a direct sum of the part on which Frobenius acts as an isomorphism and  the part on which it acts nilpotently. Denote the former by $H^1(( X/W_1)^{\llog}_{\op{crys}})_{ss}$.
We have the following maps:
	\begin{equation*}
	\begin{tikzcd}
 H^0( X,\Omega^{\ge 1}_{ X})\arrow[d,hook]  \arrow[hookrightarrow]{r} & H^1_{\op{crys}}( X/k) \arrow{r}\arrow{d} & H^1({ X,\mathcal O_{ X}})\arrow{d}\\
 H^0( X,\omega^{\ge 1}_{\llog, X})\arrow[hookrightarrow]{r} & H^1( X/W_1)^{\llog}_{\op{crys}}\arrow{r} & H^1( X,\mathcal O_{ X})
	\end{tikzcd}\hfill .
      \end{equation*}
Since $F^*$ acts by zero on the image of $ H^0( X,\Omega^{1}_{ X})$ in $H^1_{\op{crys}}(X/k)$, it acts by zero on the image of $ H^0( X,\omega^{1}_{\llog, X})$ in    $H^1( X/W_1)^{\llog}_{\op{crys}}$ and therefore $H^1(( X/W_1)^{\llog}_{\op{crys}})_{ss}\subset H^1(X,\mathcal O_{ X})_{ss}=0$,
 the latter equality being true because:
\begin{equation*}	
  H^1( X,\mathcal O_{ X})_{ss}=H^1( X,\mathcal O_{ X})^{F=1}\otimes_{\mathbb F_p}	k=H^1_{\text{\'et}}(X,\mathbb F_p)\otimes k=\operatorname{Hom}(\pi_1^{\op{ab}}( X),\mathbb F_p)\otimes k=0.
  \end{equation*}
So, there exists some $d\in \mathbb N$ s.t.  $(F^d)^*H^1(X/W_1)^{\llog}_{\op{crys}}=0$.
\end{claim}

By the same computations as in Step 1, one can actually have:
	\begin{equation*}
	\operatorname{Ker}\left(\op{MIC}^{\llog}(D_{n+m})\to \op{MIC}^{\llog}(D_n)\right)\simeq H^1( X,\omega^\bullet_{\llog,D_m}), \text{ for all } 1\le m\le n.
	\end{equation*}
        Assume now that we have $E$ as in the assumption of the theorem.
        Then $E_2\in \operatorname{Ker}\left(\op{MIC}^{\llog}(D_2)\to \op{MIC}^{\llog}(D_1)\right)$, and this defines a class $e(E_2)\in H^1(X/W_1)^{\llog}_{\op{crys}}$. By the definition of $d$ above, we have $(F^d)^*e(E_2)=e((F^*)^d E)_2=0$. Hence, by Step 1, $((F^*)^d E)_2$ is trivial. Repeating this we find $r\in \mathbb N$ such that $((F^r)^* E)_N=(F^r)^*E_N=:E'_N$ is trivial.

	We also see that the image of $E'_{2N}$ via the restriction
	$H^1( X,\omega^\bullet_{\llog,D_N})\to H^1(X,\omega^\bullet_{\llog,D_1})$, which is equal to $E'_{N+1}$ is trivial.
	
	Continuing like this, we can show that $E'_n$ is trivial for all $n\ge N$, therefore $E'$ is trivial.
	The functor $F^*$ is fully faithful restricted to locally free crystals, so we conclude that $E$ is itself trivial.
	
\end{proof}

This proves that, with notations as in the previous sections, $E=(L^{\llog})^{\otimes N}$ is trivial and therefore that a power of the rank one crystal $L$ is trivial. Using the following two lemmas, we conclude that $L$ itself is trivial as well as the rank 1 isocrystal we started with.

\begin{lemma}\label{Kummer}
With the previous assumptions, if $L^{\otimes n}=1$ with $(n,p)=1$, then $L$ is trivial. 
\end{lemma}
\begin{proof}
Since $L^{\otimes n} $ is trivial,
we have $L_{U}^{\otimes n} \simeq \mathcal{O}_{U}$ with $(n,p)=1$, hence 
\begin{equation*}
\underline{\operatorname{Spec}}\left(\bigoplus_{i=0}^{n-1}L_U^i\right)\to U
\end{equation*}
is a Kummer cover.

Since $\pi_1^{\op{tame,ab}}(U)=1$ this cover has to be trivial, so $n=1$ and $L_U\simeq \mathcal O_U$.
 Then we can choose a locally free extension $L_{X}$ that is trivial.
Applying the previous theorem to this, we get the desired result.
\end{proof}

Again with the previous notation we have the following:
\begin{lemma}\label{pPower}
  If $L^{\otimes p}=1$, $L$ is trivial.
\end{lemma}
\begin{proof}
	If $L^{\otimes p}=1$ then as coherent sheaves $L^{\otimes p}_U\simeq \mathcal O_U$, thus $M_X:=(F^*L)_U\simeq \mathcal{O}_U$. We choose a locally free extension $M_{ X}$ that is the trivial log crystal and apply Theorem \ref{triviality} ($F^*L$ is also log extendable if $L$ is). We get that $F^*L$ is trivial. But the functor $F^*:\operatorname{Crys}(U/W)\to \operatorname{Crys}(U/W)$ restricted to locally free crystals is fully faithful \cite[Ex. 7.3.4]{Ogus}. Hence $L$ itself is trivial.
\end{proof}

\begin{corollary}\label{exttrivial}
  Extensions of the trivial object by itself in $\op{I}_{\op{crys}}(X/W)^{\llog}$ are trivial. In particular, log extendable unipotent isocrystals on $U$ are constant.
\end{corollary}
\begin{proof}
  The group of log isocrystals $\mathcal E$ on $(X,M)/W$ for which the sequence
  \begin{equation}
0\to \mathcal O_{X/W}\otimes \mathbb Q \to \mathcal E\to \mathcal O_{X/W}\otimes \mathbb Q \to 0
   \end{equation} 
   is exact, is precisely $\op{Ext}^1(\mathcal O_{X/W}\otimes \mathbb Q,\mathcal O_{X/W}\otimes \mathbb Q) = H^1((X,M)/W,\mathcal O_{X/W}\otimes \mathbb Q)^{\llog}_{\op{crys}}$. However, \begin{equation}H^1((X,M)/W,\mathcal O_{X/W}\otimes \mathbb Q)^{\llog}_{\op{crys}}=H^1((X,M)/W,\mathcal O_{X/W})^{\llog}_{\op{crys}} \otimes \mathbb Q \simeq H^i_{rig}(U/K) =0. \end{equation}
   Let now $\mathcal E$ be a unipotent isocrystal on $U$, i.e. an isocrystal that admits a filtration whose associated graded quotients are extensions of the unit isocrystal by itself. If $\mathcal E$ is log extendable, then so are its associated graded quotients, which we denote by  $\mathcal E_i$. They fit into exact sequences 
   \begin{equation}\label{extlog}
0\to \mathcal O_{X/W}\otimes \mathbb Q \to \mathcal E_i \to \mathcal O_{X/W}\otimes \mathbb Q \to 0
\end{equation} in  $\op{I}_{\op{crys}}(X/W)^{\llog}$. Then we also have exact sequences in $\op{Crys}(U/W)_{\mathbb Q}$
  \begin{equation}
0\to\mathcal O_{U/W}\otimes \mathbb Q \to \mathcal E_i \to \mathcal O_{U/W}\otimes \mathbb Q \to 0.
\end{equation}
Indeed, (\ref{extlog}) being exact means that for all log PD thickenings $((V,M_V),(T,M_T))$ in the log crystalline site of $X/W$, we have
 \begin{equation}
0\to\mathcal O_T\otimes \mathbb Q \to \mathcal E_{i,T} \to \mathcal O_T\otimes \mathbb Q \to 0.
\end{equation}
The log PD thickenings  $((V,M_V),(T,M_T))$ restrict to PD thickenings $(V,T)$ on $X$ and they in turn define PD thickenings $(V\cap U,T)$ on $U$. Every open of $U$ is an open in $X$, thus we have the above exact sequence in $\op{Crys}(U/W)_{\mathbb Q}$. So $\mathcal E_i$ is itself an extension of the trivial isocrystal on $U$ by itself and its extension to $X$ is trivial. Applying the above theorem to this, we obtain that the $\mathcal E_i$ are trivial on $\op{Crys}(U/W)_{\mathbb Q}$.
  \end{proof}

\end{document}